\documentclass[11pt,dvipsnames]{extarticle}

\usepackage{amsfonts}
\usepackage{amsmath}
\usepackage{amssymb}
\usepackage{amscd}
\usepackage{amsthm}
\usepackage{mathrsfs}
\usepackage{graphicx}
\usepackage{wasysym}
\usepackage{enumerate}
\usepackage{xcolor}
\usepackage{tikz,tikz-cd}
\usepackage{geometry}
\usepackage{physics}
\usepackage[12hr]{datetime}
\usepackage{textcomp}
\usepackage[backref=page]{hyperref}
\usepackage{inputenc}
\usepackage{appendix}
\usepackage{mathtools}
\usepackage{xfrac}
\usepackage{nicefrac}
\usepackage{url}
\usepackage{environ}

\hypersetup{
	colorlinks=true, urlcolor=NavyBlue, linkcolor=Mahogany, citecolor=ForestGreen, pdfborder={0,0,0},
}

\newcommand{\mb}{\mathbb}

\newcommand{\mc}{\mathcal}
\newcommand{\mf}{\mathfrak}

\newcommand{\tsf}{\textsf}

\newcommand{\n}{\enspace}
\newcommand{\tx}{\text}

\newcommand{\ul}{\underline}

\newcommand{\supp}{\tx{supp}}

\newcommand{\spn}{\tx{span}}

\newcommand{\vst}{\vspace{0.2cm}}

\newcommand{\wt}{\widetilde}
\newcommand{\poly}{\tx{poly}}

\newcommand{\cl}{\tx{\normalfont cl}}
\newcommand{\zcl}{\tx{\normalfont Z-cl}}

\newcommand{\symcl}{\textrm{sym-cl}}

\newcommand{\LM}{\tx{\normalfont LM}}
\newcommand{\SM}{\tx{\normalfont SM}}

\newcommand{\Hilbert}{\tx{\normalfont H}}
\newcommand{\iref}[2]{(\hyperref[#2]{\ref*{#1}.\ref*{#2}})}

\newcommand{\modulo}[1]{\,(\tx{\normalfont mod }#1)}

\newcommand{\RM}{\tx{\normalfont RM}}

\newcommand{\wdeg}{\mathrm{wdeg}}
\newcommand{\diag}{\tx{\normalfont diag}}
\newcommand{\XOR}{\tx{\normalfont XOR}}
\newcommand{\coeff}{\tx{\normalfont coeff}}
\newcommand{\email}[1]{\href{mailto:#1}{\textcolor{NavyBlue}{\texttt{#1}}}}

\makeatletter
\def\thmhead@plain#1#2#3{%
	\thmname{#1}\thmnumber{\@ifnotempty{#1}{ }\@upn{#2}}%
	\thmnote{ {\the\thm@notefont\tsf{(#3)}}}}
\let\thmhead\thmhead@plain
\makeatother

\theoremstyle{definition}
\newtheorem{theorem}{Theorem}[section]
\newtheorem{fact}[theorem]{Fact}

\newtheorem*{claim*}{Claim}

\newtheorem{proposition}[theorem]{Proposition}
\newtheorem{observation}[theorem]{Observation}
\newtheorem{lemma}[theorem]{Lemma}
\newtheorem{corollary}[theorem]{Corollary}
\newtheorem{conjecture}[theorem]{Conjecture}
\newtheorem{remark}[theorem]{Remark}

\newtheorem{counterexample}[theorem]{Counterexample}
\newtheorem{question}[theorem]{Question}

\makeatletter
\NewEnviron{restatethm}[2]{%
	\protected@write\@auxout{}{%
		\string\@restatetheorem{#1}{\detokenize\expandafter{\BODY}}%
	}%
	\begin{theorem}[#2]\label{#1}\BODY\end{theorem}%
}
\newcommand{\@restatetheorem}[2]{%
	\expandafter\gdef\csname restatethm@#1\endcsname{#2}%
}
\newcommand{\restatethmnow}[2]{%
	\begingroup
	\renewcommand{\thetheorem}{\ref{#1}}%
	\begin{theorem}[#2]\csname restatethm@#1\endcsname\end{theorem}%
	\endgroup
}
\makeatother

\title{On Vanishing Properties of Polynomials on Symmetric Sets of the Boolean Cube, in Positive Characteristic}
\author{Srikanth Srinivasan\thanks{On leave from Department of Mathematics, IIT Bombay.  Supported by a startup grant from Aarhus University.}\\
	Aarhus University, Aarhus, Denmark\\
	\email{srikanth@cs.au.dk} 
	\and 
	S. Venkitesh\thanks{Department of Mathematics, IIT Bombay.  Supported by a PhD Scholarship from IRCC, IIT Bombay.}\\
	IIT Bombay, Mumbai, India\\
	\email{venkitesh.mail@gmail.com}}
\date{}
\begin{document}
	
	\maketitle
	
	\begin{abstract}
		The finite-degree Zariski (Z-) closure is a classical algebraic object, that has found a key place in several applications of \emph{the polynomial method} in combinatorics.  In this work, we characterize the finite-degree Z-closures of a subclass of symmetric sets (subsets that are invariant under permutations of coordinates) of the Boolean cube, in positive characteristic.
		
		Our results subsume multiple statements on finite-degree Z-closures that have found applications in extremal combinatorial problems, for instance, pertaining to set systems (Heged\H{u}s, Stud. Sci. Math. Hung. 2010; Heged\H{u}s, arXiv 2021), and Boolean circuits (Hr\v{u}bes et al., ICALP 2019).  Our characterization also establishes that for the subclasses of symmetric sets that we consider, the finite-degree Z-closures have low computational complexity.
		
		A key ingredient in our characterization is a new variant of finite-degree Z-closures, defined using vanishing conditions on only symmetric polynomials satisfying a degree bound.
	\end{abstract}
	
	\section{Introduction}
	
	\emph{The polynomial method} is an ever-expanding set of algebraic techniques, which broadly entails capturing combinatorial objects by algebraic means, specifically using polynomials, and then employing algebraic tools to infer their combinatorial features.  While several instances of the polynomial method have been part of the combinatorist's toolkit for decades, development of this method has received more traction in recent times, owing to several breakthroughs like \n(i)  Dvir's solution~\cite{dvir-2009-kakeya} to \emph{the finite-field Kakeya problem}, followed by an improvement by Dvir, Kopparty, Saraf, and Sudan~\cite{dvir-kopparty-saraf-2013-kakeya}, \n(ii)   Guth and Katz~\cite{guth-katz-2015-distinct-distances} proving a conjecture by Erd\"os on the lower bound for \emph{the distinct distances problem}, \n(iii) solutions to \emph{the capset problem} by Croot, Lev, and Pach~\cite{croot-lev-pach-2017-progression},  and Ellenberg and Gijswijt~\cite{ellenberg-gijswijt-2017-progression}, to name a few.  The surveys by Dvir~\cite{dvir-2012-incidence-theorems} and Tao~\cite{tao-2014-algebraic-combinatorial-geometry}, and the book by Guth~\cite{guth-2016-polynomial-method} provide detailed accounts of the polynomial method.
	
	In this article, we are interested in one of the earliest avatars of the polynomial method, which has the following basic template:
	\begin{enumerate}[(i)]
		\item  Associate a combinatorial object to a nonzero polynomial in a way that the degree of the polynomial is at most the size of the object.
		\item  Use the vanishing properties of the nonzero polynomial to assert a lower bound on its degree.  This gives a lower bound on the size of the combinatorial object.
	\end{enumerate}
	For most applications, a study of this avatar, in fact, distills to a study of a classical algebraic object -- the \emph{finite-degree Zariski closure}.  For any \(S\subseteq\{0,1\}^n\) and \(d\in\mb{N}\), the \tsf{degree-\(d\) Zariski (Z-) closure} of \(S\), denoted by \(\zcl_{n,d}(S)\), is defined to be the common zero set, in \(\{0,1\}^n\), of all polynomials with degree at most \(d\), that vanish at each point in \(S\).  This is a closure operator\footnote{A closure operator on a poset \((P,\le)\) is any map \(\cl:P\to P\) satisfying:\n(i)  \(a\le\cl(a),\,\forall\,a\in P\),\n(ii)  \(\cl(a)\le\cl(b),\,\forall\,a,b\in P,\,a\le b\), and\n(iii)  \(\cl(\cl(a))=\cl(a),\,\forall\,a\in P\).  This is a well-studied set operator.  See, for instance, Birkhoff~\cite[Chapter V, Section 1]{birkhoff-1973-lattice} for an introduction.}, and was defined by Nie and Wang~\cite{nie-wang-2015-hilbert} towards obtaining a better understanding of the applications of the polynomial method to combinatorial geometry.  However, it has been studied implicitly even earlier. (See, for instance, Wei~\cite{wei-1991-GHM}, Heijnen and Pellikaan~\cite{heijnen-pellikaan-1998-GHM-Reed-Muller}, Keevash and Sudakov~\cite{keevash-sudakov-2005-min-rank-inclusion}, and Ben-Eliezer, Hod, and Lovett~\cite{ben-eliezer-hod-lovett-2012-low-degree-polys}.)
	
	\subsection{Motivation}
	
	For any \(a,b\in\mb{Z},\,a\le b\), by an abuse of notation, we will denote the \emph{integer interval} of all integers between \(a\) and \(b\) by \([a,b]\).  We also abbreviate \([n]\coloneqq[1,n]\) for any \(n\in\mb{Z}^+\).  For any prime \(p\), the finite field with \(p\) elements is denoted by \(\mb{F}_p\).
	
	Let us begin by considering a few interesting instances of combinatorial problems solved by results on finite-degree Z-closures, that also motivate our work.  Moreover, these can be easily seen to follow the above template that we mentioned earlier.  For any \(x\in\{0,1\}^n\), the \tsf{Hamming weight} of \(x\), denoted by \(|x|\), is defined to be the number of \(i\in[n]\) such that \(x_i=1\).
	
	\begin{itemize}
		\item  Heged\H{u}s~\cite{hegedus-2010-balancing} proved a lower bound for a special case of a \emph{balancing problem} for set systems using the following lemma.
		\begin{lemma}[\cite{hegedus-2010-balancing}]\label{lem:Hegedus}
			Let \(n=4p\), where \(p\) is a prime.  If \(f(\mb{X})\in\mb{F}_p[\mb{X}]\) satisfies
			\begin{enumerate}[(i)]
				\item  \(f(x)=0\) for all \(x\in\{0,1\}^n\) with \(|x|=2p\), and
				\item  \(f(y)\ne0\) for some \(y\in\{0,1\}^n\) with \(|y|=3p\),
			\end{enumerate}
			then \(\deg(f)\ge p\).
		\end{lemma}
		We know several proofs of Lemma~\ref{lem:Hegedus} by now:  Heged\H{u}s~\cite{hegedus-2010-balancing} gave a proof using Gr\"obner basis theory,  Srinivasan (see~\cite{alon-kumar-volk-2020-unbalancing}) gave a simpler proof using Fermat's Little Theorem and linear algebra, and Alon~\cite{alon-2020-problems-extremal-combinatorics-IV} gave a proof using the Combinatorial Nullstellensatz~\cite{alon-1999-combinatorial}.
		
		\item  The following lemma was proven by Hr\v{u}bes, Ramamoorthy, Rao and Yehudayoff~\cite{hrubes-ramamoorthy-rao-2019-balancing} to solve a different version of the balancing problem.  They used this lemma to exploit a connection between balancing set systems and \emph{depth-2 threshold circuits}, which are an important class of Boolean circuits studied in the theory of computation.
		\begin{lemma}[\cite{hrubes-ramamoorthy-rao-2019-balancing}]\label{lem:HRRY}
			Let \(n=2p\), where \(p\) is a prime.  If \(f(\mb{X})\in\mb{F}_p[\mb{X}]\) satisfies
			
			\n(i)  \(f(x)=0\) for all \(x\in\{0,1\}^n\) with \(|x|=p\), and \n(ii)  \(f(0^n)\ne0\), then \(\deg(f)\ge p\).
		\end{lemma}
	
		\item  Recently, Heged\H{u}s~\cite{hegedus-2021-L-balancing} proved the following lemma, and gave a lower bound for an \emph{\(L\)-balancing problem} for set systems.
		\begin{lemma}[\cite{hegedus-2021-L-balancing}]\label{lem:Hegedus-again}
			Let \(p\) be a prime, \(n,\ell\in\mb{Z}^+\) and \(i\in[p^\ell-1,n-p^\ell+1]\).  For any \(f(\mb{X})\in\mb{F}_p[\mb{X}]\) such that \(\deg(f)\le p^\ell-1\), if \(f(x)=0\) for all \(x\in\{0,1\}^n\) with \(|x|=i\), then \(f(x)=0\) for all \(x\in\{0,1\}^n\) with \(|x|\in\{j\in[0,n]:j\equiv i\modulo{p^\ell}\}\).
		\end{lemma}
	\end{itemize}
	
	It should be noted that in each of the above results mentioned, we prescribe vanishing conditions on polynomials at all points having a fixed Hamming weight.  This naturally introduces \emph{symmetric sets} of the Boolean cube in our discussion.  Let \(\mf{S}_n\) denote the symmetric group on \([n]\).  A subset \(S\subseteq\{0,1\}^n\) is said to be \tsf{symmetric} if
	\[
	(x_1,\ldots,x_n)\in S,\,\sigma\in\mf{S}_n\quad\implies\quad(x_{\sigma(1)},\ldots,x_{\sigma(n)})\in S.
	\]
	It is easy to see that \(S\subseteq\{0,1\}^n\) is symmetric if and only if
	\[
	x\in S,\,y\in\{0,1\}^n,\,|y|=|x|\quad\implies\quad y\in S.
	\]
	Thus, symmetric sets of the Boolean cube are determined by the Hamming weights of the points in them, and therefore, there is a one-to-one correspondence between subsets \(E\subseteq[0,n]\) and symmetric sets \(\ul{E}\coloneqq\{x\in\{0,1\}^n:|x|\in E\}\).  
	
	It is easy to check that the finite-degree Z-closure of a symmetric set is a symmetric set.  So we will conveniently, wherever applicable, identify a symmetric set \(\ul{E},\,E\subseteq[0,n]\) with the set \(E\) itself; in particular, for \(E\subseteq[0,n]\) and \(d\in\mb{N}\), we will identify (and denote) the symmetric set \(\zcl_{n,d}(\ul{E})\subseteq\{0,1\}^n\) by \(\zcl_{n,d}(E)\subseteq[0,n]\).
	
	Indeed, the results mentioned above are, in fact, statements about finite-degree Z-closures of special symmetric sets.  In our notation, assuming we are working over the field \(\mb{F}_p\), Lemma~\ref{lem:Hegedus} states that \(3p\not\in\zcl_{4p,p-1}(2p)\), Lemma~\ref{lem:HRRY} states that \(0\not\in\zcl_{2p,p-1}(p)\), and Lemma~\ref{lem:Hegedus-again} states that \(\zcl_{n,p^\ell-1}(i)=\{j\in[0,n]:j\equiv i\modulo{p^\ell}\}\) for \(i\in[p^\ell-1,n-p^\ell+1]\).  In light of these results, we concern ourselves with the following question.
	\begin{question}\label{ques:main}
		Let \(\mb{F}\) be a field with positive characteristic.  Characterize (combinatorially) the finite-degree Z-closures \(\zcl_{n,d}(E)\), for all \(E\subseteq[0,n],\,d\in[0,n]\).
	\end{question}

	\subsection{Our results}
	
	Our first result subsumes Lemmas~\ref{lem:Hegedus}, \ref{lem:HRRY} and \ref{lem:Hegedus-again}.  Fix any field \(\mb{F}\) with positive characteristic \(p\).  By a \tsf{layer} in \(\{0,1\}^n\), we mean a symmetric set \(\ul{i},\,i\in[0,n]\).  We determine the finite-degree Z-closures of single layers.  This result could also be obtained from the proof techniques in Heged\H{u}s~\cite{hegedus-2010-balancing}, but we give what we believe is a simpler proof, not involving any Gr\"obner basis or Hilbert function computations, and that is similar to the proof by Srinivasan (see~\cite{alon-kumar-volk-2020-unbalancing}) for Lemma~\ref{lem:Hegedus}.
	
	For any \(E\subseteq[0,n]\) and \(\ell\in\mb{N}\), define
	\[
	E\oplus p^\ell=\bigcup_{j\in E}\{t\in[0,n]:t\equiv j\modulo{p^\ell}\}.
	\]
	For any \(d\in\mb{N}\), define \(\ell_p(d)=\lceil\log_p(d+1)\rceil\).  Thus, \(\ell_p(d)\) is the unique integer \(\ell\in\mb{N}\) such that \(p^{\ell-1}\le d\le p^\ell-1\).
	
	We have the following result, which answers Question~\ref{ques:main} for single layers.
	\begin{theorem}[Finite-degree Z-closure of a single layer]\label{thm:clo-layer-p}
		Let $i,d\in[0,n]$ and $\ell=\ell_p(d)$.  Then
		\[
		\zcl_{n,d}(i)=\begin{cases}
			\{i\},&i\not\in[d,n-d]\\
			i\oplus p^\ell ,&i\in[d,n-d]
		\end{cases}
		\]
	\end{theorem}
	
	We will then proceed to describe the finite-degree Z-closures of general symmetric sets.  We do not manage to determine these for all symmetric sets, but for a large subclass.  In this context, a variant of the finite-degree Z-closure shows itself very naturally.
	
	Since our interest lies in symmetric sets, it begs the question whether vanishing conditions on just symmetric polynomial functions would suffice to understand the finite-degree Z-closures.  Towards this, for any \(E\subseteq[0,n]\) and \(d\ge0\), we define the \tsf{degree-\(d\) symmetric closure of \(\ul{E}\)}, denoted by \(\symcl_{n,d}(\ul{E})\), to be the common zero set, in \(\{0,1\}^n\), of all symmetric polynomial functions with degree at most \(d\), that vanish at each point in \(\ul{E}\).  As in the case of finite-degree Z-closures, it is easy to see that the finite-degree symmetric closure of a symmetric set is symmetric, and so we will again identify the symmetric sets with subsets of \([0,n]\); in particular, we will identify (and denote) \(\symcl_{n,d}(\ul{E})\subseteq\{0,1\}^n\) by \(\symcl_{n,d}(E)\subseteq[0,n]\).
	
	Our second result is a characterization of Z-closures of symmetric sets in terms of their symmetric closures, under some conditions.  Thus, we answer Question~\ref{ques:main} for a special subclass of symmetric sets.
	\begin{restatethm}{thm:mod-p-main}{Finite-degree Z-closures of symmetric sets}
		Let \(d\in\mb{N},\,\ell=\ell_p(d)\).  If \(n\ge4p^\ell-1\), then for any \(E\subseteq[d,n-d]\), we have \(\zcl_{n,d}(E)=\symcl_{n,d}(E)\).
	\end{restatethm}
	
	The finite-degree symmetric closures are also interesting due to them having \emph{low computational complexity} relative to finite-degree Z-closures.  It is known that in the worst-case, computing the finite-degree Z-closure of an arbitrary subset of \(\{0,1\}^n\) will take time exponential in \(n\); in contrast, we will show by an easy linear algebraic argument that the finite-degree symmetric closure of any symmetric set in \(\{0,1\}^n\) can be computed in time polynomial in \(n\).  As a consequence, by Theorem~\ref{thm:mod-p-main}, for \(d\in\mb{N},\,\ell=\ell_p(d)\), if \(n\ge4p^\ell-1\), then for any \(E\subseteq[d,n-d]\), we can compute \(\zcl_{n,d}(E)\) in time polynomial in \(n\).

	\subsection{Related work}
	
	We note here that prior to our work, there have been attempts to characterize other notions related to finite-degree Z-closures -- namely, \emph{Gr\"obner basis}, \emph{standard monomials}, and \emph{affine Hilbert function} of the vanishing ideal -- for special cases of symmetric sets of the Boolean cube, over fields of both positive and zero characteristic.  In fact, the Gr\"obner basis and the affine Hilbert function are stronger notions than the finite-degree Z-closures.  For detailed introductions, refer for instance, Cox, Little, and O'Shea~\cite{cox-2015-ideals} -- Chapter 2 (for Gr\"obner basis), Chapter 5 (for standard monomials\footnote{The terminology `standard monomials', however, is not used in Cox, Little, and O'Shea~\cite{cox-2015-ideals}.}), and Chapter 9 (for affine Hilbert function).
	
	Let \(\mb{F}\) be a field with either positive or zero characteristic.  We will assume the basic definitions as given in~\cite{cox-2015-ideals}.  For any \(S\subseteq\{0,1\}^n\), let \(\SM_\le(S)\) denote the set of \emph{standard monomials} of the vanishing ideal of \(S\) with respect to a monomial order \(\le\).\footnote{A linear order \(\le\) on the set of all monomials in \(n\) indeterminates \(X_1,\ldots,X_n\), is a \emph{monomial order} if \n(i)  \(1\le u\) for every monomial \(u\), and \n(ii)  for monomials \(u,v\) with \(u\le v\), we have \(uw\le vw\) for every monomial \(w\).}  Further, for any \(d\in[0,n]\), let \(\Hilbert_d(S)\) denote the value of the degree-\(d\) affine Hilbert function for \(S\).  Given a monomial order \(\le\), let \(\LM_\le(P)\) denote the \emph{leading monomial} (with respect to \(\le\)) of the polynomial \(P(\mb{X})\in\mb{F}[\mb{X}]\), where \(\mb{X}=(X_1,\ldots,X_n)\) are the indeterminates.  For any \(\alpha\in\mb{N}^n\), we denote the monomial \(\mb{X}^\alpha=X_1^{\alpha_1}\cdots X_n^{\alpha_n}\).
	
	The following are basic facts that show the inter-relationships between the Gr\"obner bases, standard monomials, affine Hilbert functions, and the finite-degree Z-closures.  These follow easily from the definitions, and results from the relevant chapters in~\cite{cox-2015-ideals}.
	\begin{fact}
		\begin{enumerate}[(a)]\label{fac:basic-Grobner}
			\item  Let \(\le\) be any monomial order, and \(\mc{G}_\le(S)\) be a Gr\"obner basis of the vanishing ideal of \(S\) with respect to \(\le\).  Then for any \(S\subseteq\{0,1\}^n\),
			\[
			\SM_\le(S)=\{\mb{X}^\alpha:\mb{X}^\alpha\tx{ does not divide }\LM_\le(P),\tx{ for any }P(\mb{X})\in\mc{G}_\le(S)\}.
			\]
			
			\item  Let \(\le\) be any monomial order, and \(\mc{G}_\le(S)\) be a Gr\"obner basis of the vanishing ideal of \(S\) with respect to \(\le\).  Then for any \(S\subseteq\{0,1\}^n,\,x\in\{0,1\}^n\),
			\[
			x\in\zcl_{n,d}(S)\quad\iff\quad P(x)=0,\tx{ for all }P(\mb{X})\in\mc{G}_{\le}(S),\,\deg(P)\le d.
			\]
			
			\item  For any \(S\subseteq\{0,1\}^n,\,x\in\{0,1\}^n\), and \(d\in[0,n]\),
			\[
			x\in\zcl_{n,d}(S)\quad\iff\quad\Hilbert_d(S\cup\{x\})=\Hilbert_d(S).
			\]
		\end{enumerate}
	\end{fact}
	Some immediate corollaries, for symmetric sets of the Boolean cube, are as follows.
	\begin{corollary}
		\begin{enumerate}[(a)]\label{cor:basic-Grobner}
			\item  Let \(\le\) be any monomial order, and \(\mc{G}_\le(S)\) be a Gr\"obner basis of the vanishing ideal of \(S\) with respect to \(\le\).  Then for any \(E\subseteq[0,n],\,j\in[0,n]\),
			\[
			j\in\zcl_{n,d}(E)\quad\iff\quad P|_{\ul{j}}=0,\tx{ for all }P(\mb{X})\in\mc{G}_{\le}(\ul{E}),\,\deg(P)\le d.
			\]
			
			\item  For any \(E\subseteq[0,n],\,j\in[0,n]\), and \(d\in[0,n]\),
			\[
			j\in\zcl_{n,d}(E)\quad\iff\quad\Hilbert_d(\ul{E\cup\{j\}})=\Hilbert_d(\ul{E}).
			\]
		\end{enumerate}
	\end{corollary}

	Some of the prior work on Gr\"obner bases, standard monomials, affine Hilbert functions, and finite-degree Z-closures for symmetric sets are as follows.
	\begin{itemize}
		\item  Wilson~\cite{wilson-1990-diagonal-incidence} determined the diagonal form (over all fields) for incidence matrices associated to certain symmetric sets.  This was then used to determine the affine Hilbert function \(\Hilbert_d(\ul{i})\), for all \(d,i\in[0,n],\,i\in[d,n-d]\).
		\item  Anstee, R\'onyai, and Sali~\cite{anstee-ronyai-sali-2002-shattering} defined \emph{order shattering} for set systems, and gave a characterization of standard monomials for any subset of the Boolean cube, with respect to all lexicographic orders, in terms of order shattered sets.  Friedl and R\'onyai~\cite{friedl-ronyai-2003-order-shattering} used this characterization and generalized the result of Wilson~\cite{wilson-1990-diagonal-incidence} on incidence matrices.
		\item  Heged\H{u}s and R\'onyai~\cite{hegedus-ronyai-2003-grobner-complete-uniform} characterized the \emph{reduced} Gr\"obner basis for a single layer \(\ul{i}\), for all \(i\in[0,n]\), with respect to all lexicographic orders (over all fields), and further generalized this characterization to \emph{linear Sperner families} (over characteristic zero) in~\cite{hegedus-ronyai-2018-linear-sperner}.
		\item  Felszeghy, R\'ath, and R\'onyai~\cite{felszeghy-rath-ronyai-2006-lex} studied a \emph{lex game} to give a combinatorial criterion for a squarefree monomial to be a standard monomial of a symmetric set (over all fields).
		\item  Felszeghy, Heged\H{u}s, and R\'onyai~\cite{felszeghy-hegedus-ronyai-2009-complete-wide} obtained characterizations of a Gr\"obner basis, standard monomials, as well as the affine Hilbert function, for the symmetric set \(\ul{[d,d+\ell]\oplus p^k}\), for \(k\in\mb{Z}^+\) and \(d,\ell\in[0,n],\,d+\ell\le n\) (over positive characteristic \(p\)).
		\item  Over characteristic zero, Bernasconi and Egidi~\cite{bernasconi-egidi-1999-hilbert-symmetric} determined the affine Hilbert functions of all symmetric sets.  This can be used, via Corollary~\ref{cor:basic-Grobner} (b), to determine the finite-degree Z-closures of all symmetric sets.  Further, a more combinatorial characterization of the finite-degree Z-closures of all symmetric sets, independent of affine Hilbert function computations, was given by the second author~\cite{venkitesh-2021-covering-symmetric-sets}.
	\end{itemize}
	
	\section{Preliminaries}
	
	Since we will work over fields of positive characteristic, and since we are only concerned with subsets of the Boolean cube, we can and will assume throughout that we have fixed the field \(\mb{F}_p\), where \(p\) is prime.  So we have the Boolean cube \(\{0,1\}^n\subseteq\mb{F}_p^n\).
	
	For any set of polynomials \(\mc{P}\) in \(\mb{F}_p[\mb{X}]\), where \(\mb{X}=(X_1,\ldots,X_n)\) are the indeterminates, let \(\mc{Z}(\mc{P})=\{x\in\{0,1\}^n:P(x)=0,\tx{ for all }P(\mb{X})\in\mc{P}\}\).
	
	A fundamental result in our context is Alon's Combinatorial Nullstellensatz, which we state here for the Boolean cube.
	\begin{theorem}[\cite{alon-1999-combinatorial}]\label{thm:CN}
		The set of monomials \(\{\mb{X}^\alpha:\alpha\in\{0,1\}^n\}\) is a basis of the vector space of all \(\mb{F}_p\)-valued functions on \(\{0,1\}^n\).
	\end{theorem}
	Note that \(\mb{X}^\alpha,\,\alpha\in\{0,1\}^n\) are precisely all the squarefree monomials in the indeterminates \(\mb{X}=(X_1,\ldots,X_n)\).  So Theorem~\ref{thm:CN} implies the following:  for any polynomial \(Q(\mb{X})\in\mb{F}_p[\mb{X}]\), there exists a unique polyomial \(\widetilde{Q}(\mb{X})\in\mb{F}_p[\mb{X}]\) which is a linear combination of squarefree monomials, such that \(Q=\wt{Q}\) as functions on \(\{0,1\}^n\).  Henceforth, for convenience, we will identify \(Q(\mb{X})\) with \(\wt{Q}(\mb{X})\); in other words, for any polynomial \(Q(\mb{X})\) that we define, that is not necessarily a linear combination of squarefree monomials, we will assume that \(Q(\mb{X})\) has been immediately replaced by \(\wt{Q}(\mb{X})\), and denoted by \(Q(\mb{X})\) itself, without mention.  Also relevant is that, as a consequence, while considering the finite-degree Z-closure \(\zcl_{n,d}\), we can restrict \(d\in[0,n]\).
	
	\paragraph*{Finite-degree Z-closure.}  For any \(S\subseteq\{0,1\}^n\), let \(\mc{I}_{n,d}(S)\) denote the vector space of all polynomials in \(\mb{F}_p[\mb{X}]\) having degree at most \(d\), that vanish at each point in \(S\).  Recall that for any \(d\in[0,n]\) and \(S\subseteq\{0,1\}^n\), the \tsf{degree-\(d\) Zariski (Z-) closure} is defined by \(\zcl_{n,d}(S)=\mc{Z}(\mc{I}_{n,d}(S))\).  It is easy to check that for any \(E\subseteq[0,n]\), \(\zcl_{n,d}(\ul{E})\) is a symmetric set.  So we will stick to our identification of symmetric sets of \(\{0,1\}^n\) with subsets of \([0,n]\), and use the notation \(\zcl_{n,d}(E)\) instead.  For any \(E\subseteq\mb{Z}\) and \(a,b\in\mb{Z}\), define \(a+bE=\{a+bx:x\in E\}\).  We make the following preliminary observations.
	\begin{proposition}[Properties of finite-degree Z-closures of symmetric sets]\label{pro:clo-prop}
		
		Consider any \(d\in[0,n]\) and \(E\subseteq[0,n]\).
		\begin{enumerate}[(a)]
			\item  If $j\in\zcl_{n,d}(E)$, then $j\in\zcl_{m,d}(E)$, for all $m>n$.
			\item  If $j\in\zcl_{n,d}(E)$, then $n-j\in\zcl_{n,d}(n-E)$.
			\item  If $j\in\zcl_{n,d}(E)$, then $j+k\in\zcl_{n+k,d}(E+k)$, for all $k>0$.
		\end{enumerate}
	\end{proposition}
	
	\begin{proof}
		\begin{enumerate}[(a)]
			\item  Let $j\in\zcl_{n,d}(E)$ and $m>n$.  It is enough to show that $1^j0^{m-j}\in\zcl_{m,d}(E)$.  Consider any $f(X_1,\ldots,X_m)\in\mc{I}_{m,d}(E)$.  Define $f^*(X_1,\ldots,X_n)=f(X_1,\ldots,X_n,0^{n-m})$.  Note that for any $x\in\ul{E}$ in $\{0,1\}^n$, we have $x0^{n-m}\in\ul{E}$ in $\{0,1\}^m$.  Also $\deg f^*\le\deg f\le d$ and hence $f^*(X_1,\ldots,X_n)\in\mc{I}_{n,d}(E)$.  Then $f(1^j0^{m-j})=f^*(1^j0^{n-j})=0$, since $j\in\zcl_{n,d}(E)$.  Thus $1^j0^{m-j}\in\zcl_{m,d}(E)$.
			
			\item  Let \(j\in\zcl_{n,d}(E)\).  Consider any \(f(X_1,\ldots,X_n)\in\mc{I}_{n,d}(n-E)\).  It is enough to show that \(1^{n-j}0^j\in\zcl_{n,d}(n-E)\).  Define \(f^*(X_1,\ldots,X_n)=f(1-X_1,\ldots,1-X_n)\).  Then we have \(f(x_1,\ldots,x_n)=0\), where \((x_1,\ldots,x_n)\in\ul{n-E}\) if and only if \(f^*(1-x_1,\ldots,1-x_n)=0\), where \((1-x_1,\ldots,1-x_n)\in\ul{E}\).  So \(f^*(X_1,\ldots,X_n)\in\mc{I}_{n,d}(E)\).  This gives \(f^*(0^{n-j}1^j)=0\), since \(j\in\zcl_{n,d}(E)\).  Thus \(f(1^{n-j}0^j)=0\), that is, \(1^{n-j}0^j\in\zcl_{n,d}(n-E)\).
			
			\item  Let $j\in\zcl_{n,d}(E)$ and \(k>0\).  It is enough to show that $1^j0^{n-j}1^k\in\zcl_{n+k,d}(E+k)$.  Consider any $f\in\mc{I}_{n+k,d}(E+k)$.  Define $f^*(x)=f(x1^k)$, for all $x\in\{0,1\}^n$.  Note that for any $x\in\ul{E}$ in $\{0,1\}^n$, we have $x1^k\in\ul{E+k}$ in $\{0,1\}^{n+k}$.  Also $\deg f^*\le\deg f\le d$ and hence $f^*\in\mc{I}_{n,d}(E)$.  Then $f(1^j0^{n-j}1^k)=f^*(1^j0^{n-j})=0$, since $j\in\zcl_{n,d}(E)$.  Thus $1^j0^{n-j}1^k\in\zcl_{n+k,d}(E+k)$.\qedhere
		\end{enumerate}
	\end{proof}

	\paragraph*{\(p\)-ary representation of nonnegative integers.}  For any \(d\in\mb{N}\), define \(\ell_p(d)=\lceil\log_p(d+1)\rceil\).  Thus, \(\ell_p(d)\) is the unique integer \(\ell\in\mb{N}\) such that \(p^{\ell-1}\le d\le p^\ell-1\).  For any \(n\in\mb{N}\), we fix the following notation via the \(p\)-ary expansion of \(n\),
	\[
	n=\sum_{t\ge0}n_tp^t,\quad\tx{where }n_t\in[0,p-1],\tx{ for all }t\ge0.
	\]
	In other words, \(n_t\) denotes the \(t\)-th digit of \(n\) in its \(p\)-ary expansion, for all \(t\ge0\).  The following observations are immediate.
	\begin{observation}\label{obs:p-ary-prop}
		Let \(m,n,\ell\in\mb{N}\).
		\begin{enumerate}[(a)]
			\item  \(m=n\) if and only if \(m_t=n_t\), for all \(t\ge0\).
			\item  If \(m<n\), then there exists \(t\in[0,\ell_p(n)-1]\) such that \(m_t<n_t\).
			\item  \(m\equiv n\modulo{p^\ell}\) if and only if \(m_t=n_t\), for all \(t\in[0,\ell-1]\).
		\end{enumerate}
	\end{observation}
	
	\paragraph*{Elementary symmetric polynomials.}  Fix \(n\in\mb{Z}^+\).  For \(k\in[0,n]\), the \tsf{elementary symmetric polynomial} of degree \(k\) is a multilinear polynomial of degree \(k\) defined as
	\[
	\sigma_k(\mb{X})=\sum_{\substack{S\subseteq[n]\\|S|=k}}\prod_{i\in S}X_i\in\mb{F}_p[\mb{X}].
	\]
	It follows immediately from the definition that for any \(k\in[0,n]\) and \(x\in\{0,1\}^n\), we have \(\sigma_k(x)\equiv\binom{|x|}{k}\modulo{p}\).
	
	The following result is crucial for us to work with the elementary symmetric polynomials.
	\begin{theorem}[Lucas's Theorem~\cite{lucas-1878-II}]\label{thm:Lucas}
		For any \(n,m\in\mb{N}\),
		\[
		\binom{n}{m}\equiv\prod_{t\ge0}\binom{n_t}{m_t}\modulo{p}.
		\]
		So \(\binom{n}{m}\ne0\modulo{p}\) if and only if \(m_t\le n_t\) for all \(t\ge0\).
	\end{theorem}
	As an immediate corollary, we get some properties of elementary symmetric polynomials.
	\begin{corollary}[Properties of elementary symmetric polynomials]\label{cor:sym-digit}
		
		\phantom{}
		
		\begin{enumerate}[(a)]
			\item  For any \(t\ge0\) and \(x\in\{0,1\}^n\), we have \(\sigma_{p^t}(x)\equiv|x|_t\modulo{p}\).
			
			\item  For any \(x,y\in\{0,1\}^n\), if \(|y|\equiv|x|\modulo{p^t}\) for some \(t\ge0\), then \(\sigma_d(x)\equiv\sigma_d(y)\modulo{p}\), for all \(d\in[0,p^t-1]\).
		\end{enumerate}
	\end{corollary}
	\begin{proof}
		\begin{enumerate}[(a)]
			\item  By definition, we get
			\begin{align*}
				\sigma_{p^t}(x)&\equiv\binom{|x|}{p^t}\modulo{p}&\\
				&\equiv\binom{\sum_{k\ge0}|x|_kp^k}{\sum_{k\ge0}(p^t)_k}\modulo{p}&\\
				&\equiv\binom{|x|_t}{1}\modulo{p}&\tx{by Theorem~\ref{thm:Lucas}}\\
				&\equiv|x|_t\modulo{p}.&
			\end{align*}
			
			\item  Let \(x,y\in\{0,1\}^n\) such that \(|y|\equiv|x|\modulo{p^t}\).  So we have \(|y|_k=|x|_k\), for all \(k\in[0,t-1]\).  Further, for any \(d\in[0,p^t-1]\), we have \(d_k=0\) for all \(k\ge t\).  So by definition, we get
			\begin{align*}
				\sigma_d(y)&\equiv\binom{|y|}{d}\modulo{p}&&\\
				&\equiv\prod_{k\ge0}\binom{|y|_k}{d_k}\modulo{p}&\tx{by Theorem~\ref{thm:Lucas}}&\\
				&\equiv\prod_{k=0}^{t-1}\binom{|y|_k}{d_k}\modulo{p}&&\\
				&\equiv\prod_{k=0}^{t-1}\binom{|x|_k}{d_k}\modulo{p}&&\\
				&\equiv\binom{|x|}{d}\modulo{p}&\tx{by Theorem~\ref{thm:Lucas}}&\\
				&\equiv\sigma_d(x)\modulo{p}.&&\qedhere
			\end{align*}
		\end{enumerate}
	\end{proof}

\paragraph*{Integer-valued polynomials.}  The integer-valued polynomials\footnote{See for instance, Cahen and Chabert~\cite{cahen-chabert-1997-integer-polynomials} for a detailed account of integer-valued polynomials.} are precisely those polynomials \(P(Z)\in\mb{Q}[Z]\) such that \(P(\mb{Z})\subseteq\mb{Z}\).  For any \(k\in\mb{N}\), consider the \tsf{degree-\(k\) Newton polynomial} defined as \(\binom{Z}{k}=(1/k!)\cdot Z(Z-1)\cdots(Z-k+1)\in\mb{Z}[Z]\).  It is clear that as a function on \(\mb{Z}\), \(\binom{Z}{k}\) is \(\mb{Z}\)-valued, for all \(k\in\mb{N}\).  The following lemma is folklore.  The first mention of this result could be attributed to a letter by James Gregory to John Collins dated November 23, 1670~\cite{turnbull-1959-newton-correspondences}.	(See, for instance, Cahen and Chabert~\cite[Corollary I.1.2]{cahen-chabert-1997-integer-polynomials}.)
	\begin{lemma}[Folklore,~{\cite{turnbull-1959-newton-correspondences},\cite[Corollary I.1.2]{cahen-chabert-1997-integer-polynomials}}]\label{lem:Newton-AW}
		Let \(d\in\mb{N}\) and \(I\subseteq\mb{N}\) be an interval with \(|I|=d+1\).  For any function \(f:I\to\mb{N}\), there exists a unique polynomial \(Q_f(Z)\) that is a \(\mb{Z}\)-linear combination of \(\binom{Z}{0},\ldots,\binom{Z}{d}\) such that \(Q_f=f\).
	\end{lemma}

	Finally, yet another abbreviation.  For any \(a_1,\ldots,a_k\in\{0,1\}\) and \(n_1,\ldots,n_k\in\mb{N}\), we denote the binary vector \(a_1^{n_1}\cdots a_k^{n_k}\coloneqq\big(\underbrace{a_1,\ldots,a_1}_{n_1\tx{ times }},\ldots,\underbrace{a_k,\ldots,a_k}_{n_k\tx{ times}}\big)\in\{0,1\}^{n_1+\cdots+n_k}\).

	\section{Finite-degree Z-closure of a single layer}
	
	In this section, we will prove Theorem~\ref{thm:clo-layer-p}.  Towards this, the following is an easy but important proposition.
	\begin{proposition}\label{pro:layer-obs}
		Let \(d\in[0,n]\) and \(\ell=\ell_p(d)\).
		\begin{enumerate}[(a)]
			\item  If \(i\not\in[d,n-d]\), then \(\zcl_{n,d}(i)=\{i\}\).
			\item  If \(i\in[d,n-d]\), then \(\zcl_{n,d}(i)\subseteq i\oplus p^\ell \).
		\end{enumerate}
	\end{proposition}
	\begin{proof}
		\begin{enumerate}[(a)]
			\item  Suppose $i<d$.  For any $j>i$, the multilinear polynomial $X_1\cdots X_{i+1}$ vanishes on $\ul{i}$, does not vanish at $1^j0^{n-j}\in\ul{j}$ and has degree at most \(d\).  So $j\not\in\zcl_{n,d}(i)$.  Now consider any $j<i$.  Then there exists $t\in[0,\ell-1]$ such that $j_t<i_t$.  Then $\sigma_{p^t}(X_1,\ldots,X_n)-i_t$ is a multilinear polynomial of degree $p^t\le p^{\ell-1}\le d$, which is zero on $\ul{i}$ and nonzero on $\ul{j}$, since $j_t\ne i_t$.  So $j\not\in\zcl_{n,d}(i)$.  Thus $\zcl_{n,d}(i)=\{i\}$, if $i<d$.  By Proposition~\ref{pro:clo-prop} (b), we then also get $\zcl_{n,d}(i)=\{i\}$, if $i>n-d$.  Hence $\zcl_{n,d}(i)=\{i\}$.
			
			\item  Consider any \(j\not\in i\oplus p^\ell \).  So \(j\not\equiv i\modulo{p^\ell }\).  Then there exists $t\in[0,\ell-1]$ such that $j_t\ne i_t$.  Then $\sigma_{p^t}(X_1,\ldots,X_n)-i_t$ is a multilinear polynomial which is zero on $\ul{i}$ and nonzero on $\ul{j}$, with degree at most $p^t\le p^{\ell-1}\le d$.  So $j\not\in\zcl_{n,d}(i)$.  Thus $\zcl_{n,d}(i)\subseteq i\oplus p^\ell$.\qedhere
		\end{enumerate}
	\end{proof}
	
	We will also need the following technical lemma.
	\begin{lemma}\label{lem:hr}
		Let \(n\in\mb{N},\,i\in[0,n]\), and \(\ell=\ell_p(i)\).
		\begin{enumerate}[(a)]
			\item  There exists \(h_i(\mb{X})\in\mb{F}_p[\mb{X}]\) such that \(\deg h_i\le p^\ell -1\) and
			\[
			h_i(x)=0\quad\iff\quad|x|\not\equiv i\modulo{p^\ell }.
			\]
			
			\item  For any \(j\in i\oplus p^\ell,\,j>i\), there exists \(r_{i,j}(\mb{X})\in\mb{F}_p[\mb{X}]\) such that \(\deg r_{i,j}\le j-i-p^\ell\) and
			\[
			r_{i,j}(x)\begin{cases}
				=0,&|x|\in[i+1,j-1],\,|x|\equiv i\modulo{p^\ell }\\
				\ne0,&|x|=j
			\end{cases}
			\]
		\end{enumerate}
	\end{lemma}
	Let us first assume Lemma~\ref{lem:hr} and prove Theorem~\ref{thm:clo-layer-p}.  We will also need a result characterizing the duals of Reed-Muller codes over the Boolean cube.
	
	The duals of Reed-Muller codes over \(\mb{F}_q^n\) (where \(q\) is a power of \(p\)) were determined by Delsarte, Goethals, and MacWilliams~\cite{delsarte-goethals-macwilliams-1970-generalized-RM}, who remark that it could be readily obtained from Kasami, Lin, and Peterson~\cite{kasami-lin-peterson-1968-generalization-RM}, and is also mentioned in some unpublished notes of Lin~\cite{lin-unpublished-notes-GRM}.  Beelen and Datta~\cite{beelen-datta-2018-GHM-RM}, using a different argument, described these duals more generally over finite grids in \(\mb{F}_q^n\).  The characterization of these duals over the Boolean cube could be obtained by the proofs in either of the above works; in fact, it is a special case of~\cite[Theorem 5.7]{beelen-datta-2018-GHM-RM}.
	
	For \(d\in[0,n]\), the \tsf{Reed-Muller code} (over the Boolean cube \(\{0,1\}^n\subseteq\mb{F}_p^n\)) with degree parameter \(d\) is defined as
	\[
	\RM_p(n,d)=\Big\{P\coloneqq\begin{bmatrix}
		P(a)
	\end{bmatrix}_{a\in\{0,1\}^n}:P(\mb{X})\in\spn\{\mb{X}^\alpha:\alpha\in\{0,1\}^n\},\,\deg P\le d\Big\}.
	\]
	\begin{theorem}[Dual of Reed-Muller code~{\cite{delsarte-goethals-macwilliams-1970-generalized-RM,kasami-lin-peterson-1968-generalization-RM,lin-unpublished-notes-GRM},\cite[Theorem 5.7]{beelen-datta-2018-GHM-RM}}]\label{thm:RM-usual}
		\phantom{}
		
		For any \(d\in[0,n]\),
		\begin{align*}
			\RM_p(n,d)^\perp&=\Big\{\begin{bmatrix}
				(-1)^{|a|}Q(a)
			\end{bmatrix}_{a\in\{0,1\}^n}:Q(\mb{X})\in\spn\{\mb{X}^\alpha:\alpha\in\{0,1\}^n\},\,\deg Q\le n-d-1\Big\}\\
			&=\Big\{\diag\big((-1)^{|a|}:a\in\{0,1\}^n\big)\cdot\begin{bmatrix}
				Q(a)
			\end{bmatrix}_{a\in\{0,1\}^n}:Q\in\RM_p(n,n-d-1)\Big\}.
		\end{align*}
	\end{theorem}

	 For any \(v\in\mb{F}_p^n\), define \(\supp(v)=\{i\in[n]:v_i\ne0\}\).  Similarly, for any function \(f:A\to\mb{F}_p\) (where \(A\) is some set), define \(\supp(f)=\{a\in A:f(a)\ne0\}\).
	 \begin{remark}\label{rem:GRM-RM-support}
	 	Note that for any \(\lambda_i\in\mb{F}_p\setminus\{0\},\,i\in[n]\), and \((v_1,\ldots,v_n)\in\mb{F}_p^n\), we have \(\supp(v_1,\ldots,v_n)=\supp(\lambda_1v_1,\ldots,\lambda_nv_n)\).
	 \end{remark}
	
	Consider the standard \emph{dot product} for vectors in \(\mb{F}_p^n\):  \(v\cdot w=\sum_{i\in[n]}v_iw_i\) for \(v,w\in\mb{F}_p^n\).  The following is a standard fact from linear algebra, which follows from the properties of duals of linear subspaces with respect to the dot product.\footnote{Note that in characteristic \(p\), the dot product is not an \emph{inner product}.  In fact, it is not even nondegenerate; there exist \(v\in\mb{F}_p^n,\,v\ne0\) such that \(v\cdot v=0\).  But nevertheless, it is a bilinear form, with respect to which we can consider dual subspaces.}
	\begin{fact}\label{fact:dual-code}
		Let \(W\subseteq\mb{F}_p^n\) be a linear subspace, and \(S\subseteq[n],\,j\in[n]\).  The following are equivalent.
		\begin{itemize}
			\item  For any \(w\in W\), if \(w_i=0\) for all \(i\in S\), then \(w_j=0\).
			\item  There exists \(v\in W^\perp\) such that \(j\in\supp(v)\subseteq\{j\}\cup S\).
		\end{itemize}
	\end{fact}
	
	We are now ready to prove Theorem~\ref{thm:clo-layer-p}.
	\begin{proof}[Proof of Theorem~\ref{thm:clo-layer-p}]
		If \(i\not\in[d,n-d]\), then we are done by Proposition~\ref{pro:layer-obs} (a).
		
		Now suppose $i\in[d,n-d]$.  By Proposition~\ref{pro:layer-obs} (b), we know that \(\zcl_{n,d}(i)\subseteq i\oplus p^\ell \).  We now prove that \(i\oplus p^\ell \subseteq\zcl_{n,d}(i)\).  By Proposition~\ref{pro:clo-prop} (b), it is enough to prove that $j\in\zcl_{n,d}(i)$, for all $j\in i\oplus p^\ell ,\,j>i$.	 Further, by Proposition~\ref{pro:clo-prop} (c), it is enough to consider $i=d,\,j=d+p^\ell k$.  Furthermore, by Proposition~\ref{pro:clo-prop} (a), it is enough to consider \(i=d,\,n=j=d+p^\ell k\).  Therefore, we need to prove that \(n\in\zcl_{n,d}(d)\).
		
		Consider \(h_d(\mb{X}),\,r_{d,n}(\mb{X})\in\mb{F}_p[\mb{X}]\) as given by Lemma~\ref{lem:hr}.  So we have
		\begin{align}
			&h_d(x)=0\quad\iff\quad|x|\not\equiv d\modulo{p^\ell}\label{eq:h}\\
			\tx{and}\quad&r_{d,n}(x)\begin{cases}
				=0,&|x|\in[d+1,n-1],\,|x|\equiv d\modulo{p^\ell}\\
				\ne0,&x=1^n
			\end{cases}\label{eq:r}
		\end{align}
		and further, \(\deg h_d\le p^\ell-1\) and \(\deg r_{d,n}\le n-d-p^\ell\).  We now give two (essentially equivalent) arguments.
		\begin{enumerate}[(a)]
			\item  We first note the following sequence of equivalences.
			\begin{align*}
				&1^n\in\zcl_{n,d}(\ul{d})&&\\
				\iff&\not\exists f\in \RM_p(n,d):f(1^n)\ne0,\,f|_{\ul{d}}=0&&\\
				\iff&\exists\,g\in \RM_p(n,d)^\perp:1^n\in\supp(g)\subseteq\{1^n\}\cup\ul{d}&\tx{by Fact~\ref{fact:dual-code}}&\\
				\iff&\exists\,g\in \RM_p(n,n-d-1):1^n\in\supp(g)\subseteq\{1^n\}\cup\ul{d}&\qquad\tx{by Theorem~\ref{thm:RM-usual} and Remark~\ref{rem:GRM-RM-support}}&
			\end{align*}
			Now define \(g(\mb{X})=h_d(\mb{X})r_{d,n}(\mb{X})\).  Then $\deg g=\deg h_d+\deg r_{d,n}\le(p^\ell-1)+(n-d-p^\ell)=n-d-1$.  So \(g\in\RM_p(n,n-d-1)\).  Further, by (\ref{eq:h}) and (\ref{eq:r}), we have \(1^n\in\supp(g)\subseteq\{1^n\}\cup\ul{d}\).  Thus \(n\in\zcl_{n,d}(d)\).
			
			\item  Suppose \(n\not\in\zcl_{n,d}(d)\).  Let \(P(\mb{X})\in\mb{F}_p[\mb{X}]\) be a polynomial such that \(\deg P\le d,\,P|_{\ul{d}}=0\) and \(P(1^n)\ne0\).  Define \(Q(\mb{X})=P(\mb{X})h_d(\mb{X})r_{d,n}(\mb{X})\).  Then \(Q(x)\ne0\) if and only if \(x=1^n\).  So by the Combinatorial Nullstellensatz (Theorem~\ref{thm:CN}),
			\[
			Q(\mb{X})=\alpha\prod_{t\in[n]}X_t,\quad\tx{for some }\alpha\ne0,
			\]
			and so \(\deg Q=n\).  But
			\[
			\deg Q\le\deg P+\deg h_d+\deg r_{d,n}\le d+(p^\ell-1)+(n-d-p^\ell)=n-1,
			\]
			a contradiction.  Thus \(n\in\zcl_{n,d}(d)\).\qedhere
		\end{enumerate}
	\end{proof}

	\subsection{Proof of Lemma~\ref{lem:hr}}
	
	Finally, we close this section by proving Lemma~\ref{lem:hr}.
	\begin{proof}[Proof of Lemma~\ref{lem:hr}]
		\begin{enumerate}[(a)]
			\item  Let \(i\in[0,n]\) and \(\ell=\ell_p(i)\).  Define
			\[
			h_i(\mb{X})=\prod_{t=0}^{\ell-1}(1-(\sigma_{p^t}(x)-i_t)^{p-1})\in\mb{F}_p[\mb{X}].
			\]
			Then \(\deg h_i=(p-1)\sum_{t=0}^{\ell-1}p^t=p^\ell -1\).  Further, it is clear from the definition that \(h_i(x)=0\) if and only if \(|x|\not\equiv i\modulo{p^\ell }\).
			
			\item  Let \(j\in i\oplus p^\ell,\,j>i\), and so \(j=i+p^\ell k\), for some \(k\in\mb{Z}^+\).  Let \(I=\{i,i+p^\ell,\ldots,i+p^\ell k=j\}\).  Also, let \(m=\ell_p(j)\).  Then we have, for every \(r\in I\),
			\begin{align}
				&r_t\begin{cases}
					=i_t,&t\in[0,\ell-1]\\
					\in[0,p-1],&t\in[\ell,m-1]\\
					=0,&t\ge m
				\end{cases}\label{eq:hat-r}
			\end{align}
		
			By Lemma~\ref{lem:Newton-AW}, there exists a unique polynomial \(Q(Z)\) which is a \(\mb{Z}\)-linear combination of \(\binom{Z}{0},\ldots,\binom{Z}{k-1}\) such that \(Q(z)=0\) for all \(z\in[1,k-1]\), and \(Q(k)=1\).  For any \(x\in\ul{I}\), let \(|x|'\coloneqq\sum_{t=\ell}^{m-1}|x|_tp^{t-\ell}\); then we have \(|x|'_t=|x|_{\ell+t},\,t\in[0,m-\ell-1]\).  Now for any \(u\in[0,k-1]\) and \(x\in\ul{I}\), we have
			\begin{align*}
			\binom{|x|'}{u}&\equiv\prod_{t=0}^{m-\ell-1}\binom{|x|'_t}{u_t}\modulo{p}&\tx{by Theorem~\ref{thm:Lucas}}&\\
			&\equiv\prod_{t=0}^{m-\ell-1}\binom{|x|_{\ell+t}}{(p^\ell u)_{\ell+t}}\modulo{p}&\tx{since }u_t=(p^\ell u)_{\ell+t},\,t\in[0,m-\ell-1]&\\
			&\equiv\prod_{t=0}^{\ell-1}\binom{|x|_t}{(p^\ell u)_t}\cdot\prod_{t=\ell}^{m-1}\binom{|x|_t}{(p^\ell u)_t}\modulo{p}&\tx{since }(p^\ell u)_t=0,\,t\in[0,\ell-1]&\\
			&\equiv\binom{|x|}{p^\ell u}\modulo{p}&\tx{by Theorem~\ref{thm:Lucas}}&\\
			&\equiv\sigma_{p^\ell u}(x)\modulo{p}.
			\end{align*}
			Suppose \(Q(Z)=\sum_{u=0}^{k-1}c_u\binom{Z}{u}\), where \(c_0,\ldots,c_{k-1}\in\mb{Z}\).  Define \(r_{i,j}(\mb{X})=\sum_{u=0}^{k-1}c_u\sigma_{p^\ell u}(\mb{X})\in\mb{F}_p[\mb{X}]\).  Clearly, \(\deg r_{i,j}=p^\ell(k-1)=j-i-p^\ell\).
			
			Note that for any \(x\in\ul{I}\) and \(v\in[0,k]\), we have \(|x|=i+p^\ell v\) if and only if \(|x|'=v\).  So, for any \(x\in\ul{I}\), we have
			\[
			r_{i,j}(x)=\sum_{u=0}^{k-1}c_u\sigma_{p^\ell u}(x)=\sum_{u=0}^{k-1}c_u\binom{|x|'}{u}=Q(|x|')=\begin{cases}
				0,&|x|\in[i+1,j-1]\\
				1,&|x|=j
			\end{cases}
			\]
			This completes the proof.\qedhere
		\end{enumerate}
	\end{proof}

	\section{Finite-degree Z-closures of arbitrary symmetric sets}
	
	In this section, we will proceed to show that the finite-degree Z-closures of symmetric sets are equal to the symmetric closures, under some conditions.  This would imply that these Z-closures can be computed in polynomial time.  However, a \emph{more explicit} description of the finite-degree Z-closures of arbitrary symmetric sets \`a la Theorem~\ref{thm:clo-layer-p} is still in want.
	
	\subsection{Some observations, a conjecture, and the main theorem}
	
	It is interesting to note that, in characteristic zero (that is, over the field \(\mb{R}\)), for any symmetric set \(E\subseteq[0,n]\), the event \(j\not\in\zcl_{n,d}(E)\) is \emph{witnessed} by a polynomial \(P(\mb{X})\in\mb{R}[\mb{X}]\) having a special form.  This is immediate from the characterization of finite-degree Z-closures of symmetric sets, over \(\mb{R}\), given by the second author~\cite{venkitesh-2021-covering-symmetric-sets}.  Let \(\XOR=\{X_i-X_j:i,j\in[n],\,i\ne j\}\).
	\begin{observation}[Follows from~\cite{venkitesh-2021-covering-symmetric-sets}]\label{obs:layer-real}
		Let \(d\in[0,n]\) and \(E\subseteq[0,n]\).  Over the reals, if \(j\not\in\zcl_{n,d}(E)\), then there exists a polynomial \(P(\mb{X})=\ell_1(\mb{X})\cdots\ell_k(\mb{X})\sigma(\mb{X})\in\mb{R}[\mb{X}]\), where \(\ell_1,\ldots,\ell_k\in\XOR\) and \(\sigma\) is a symmetric polynomial, such that \(\deg P\le d\), \(P|_{\ul{i}}=0\) and \(P|_{\ul{j}}\ne0\).
	\end{observation}
	What we see now is that we can infer such statements, of the kind of Observation~\ref{obs:layer-real}, for Z-closures over \(\mb{F}_p\) as well.  More precisely, from Theorem~\ref{thm:clo-layer-p} and the proof of Proposition~\ref{pro:layer-obs}, we have the following.  For any \(S,T\subseteq[n],\,S\cap T=\emptyset\), define a \tsf{generalized monomial} to be the polynomial \(\mb{X}^{(S,T)}\coloneqq\prod_{s\in S}X_s\prod_{t\in T}(1-X_t)\).
	\begin{observation}\label{obs:layer-symm=Z}
		Let \(d\in[0,n],\,\ell=\ell_p(d)\).
		\begin{enumerate}[(a)]
			\item  Let \(i\in[d,n-d]\).  If \(j\not\in\zcl_{n,d}(i)\), then there exists a symmetric polynomial \(\sigma(\mb{X})\in\mb{F}_p[\mb{X}]\) such that \(\deg\sigma\le d\), \(\sigma|_{\ul{i}}=0\) and \(\sigma|_{\ul{j}}\ne0\).  In other words, \(\zcl_{n,d}(i)=\symcl_d(n,i)\).
			\item  Let \(i\not\in[d,n-d]\).  If \(j\not\in\zcl_{n,d}(i)\), then there exists a polynomial \(P(\mb{X})\in\mb{F}_p[\mb{X}]\) that is either symmetric or a generalized monomial, such that \(\deg P\le d\), \(P|_{\ul{i}}=0\) and \(P|_{\ul{j}}\ne0\).
		\end{enumerate}
	\end{observation}
	Inspired by Observation~\ref{obs:layer-symm=Z}, a reasonable conjecture about Z-closures of symmetric sets over \(\mb{F}_p\) is the following.
	\begin{conjecture}\label{conj:Z-cl-general}
		Let \(d\in[0,n]\) and \(E\subseteq[0,n]\).  For \(j\in[0,n]\), if \(j\not\in\zcl_{n,d}(E)\), then there exists a polynomial \(P(\mb{X})=m_1(\mb{X})\cdots m_k(\mb{X})\sigma(\mb{X})\in\mb{F}_p[\mb{X}]\), where \(m_1(\mb{X}),\ldots,m_k(\mb{X})\) are generalized monomials and \(\sigma(\mb{X})\) is a symmetric polynomial, such that \(\deg P\le d\), \(P|_{\ul{i}}=0\) and \(P|_{\ul{j}}\ne0\).
	\end{conjecture}
	Unfortunately, Conjecture~\ref{conj:Z-cl-general} is not true; the following is a counterexample.
	
	\begin{counterexample}\label{counterex:not-in-cl}
		We see that over \(\mb{F}_2\), \(0\not\in\zcl_{5,2}(\{1,4\})\).  A witness polynomial is the degree-2 polynomial
		\[
		(1+X_1+X_2+X_3+X_4)(1+X_2+X_3+X_4+X_5),
		\]
		which vanishes on \(\ul{\{1,4\}}\), but does not vanish at \(0^5\).  However, this witness polynomial is not of the form claimed in Conjecture~\ref{conj:Z-cl-general}.  Let us now show that there is no witness polynomial of the claimed form.  The following cover all the possibilities of a potential witness polynomial of the claimed form.
		\begin{itemize}
			\item  Firstly, it is easy to see that there is no nonzero polynomial of degree at most 1 that vanishes on \(\ul{\{1,4\}}\), and does not vanish at \(0^5\).  So a potential witness polynomial must have degree 2.
			\item  A generalized monomial of degree 2 has zero set of the form \(\{x\in\{0,1\}^n:x_i=a\tx{ or }x_j=b\}\), for \(i,j\in[5],\,i\ne j\) and \(a,b\in\{0,1\}\), which clearly does not contain \(\ul{\{1,4\}}\).  A product of two generalized monomials of degree 1, which is not a generalized monomial of degree 2, is simply \(X_i(1-X_i)\) for some \(i\in[5]\), which vanishes everywhere.
			\item  A product of a generalized monomial of degree 1 and a symmetric polynomial of degree 1 has zero set of the form \(\{x\in\{0,1\}^5:x_i=a\tx{ or }|x|=b\modulo{2}\}\), for \(i\in[5],\,a,b\in\{0,1\}\), which again does not contain \(\ul{\{1,4\}}\).
			\item  A product of two distinct symmetric polynomial of degree 1 will vanish everywhere.  The only other possibility of a  symmetric polynomial of degree 2 has zero set of the form \(\{x\in\{0,1\}^5:|x|_1=a\}\), for \(a\in\{0,1\}\), which either contains \(\ul{\{0,1,4\}}\) or does not contains \(\ul{\{1,4\}}\).
		\end{itemize}
		This completes the analysis.
	\end{counterexample}

	Nevertheless, the form of the above witness polynomial motivates the following weaker conjecture, which we leave open.
	\begin{conjecture}\label{conj:Z-cl-general-weaker}
		Let \(d\in[0,n]\) and \(E\subseteq[0,n]\).  For \(j\in[0,n]\), if \(j\not\in\zcl_{n,d}(E)\), then there exists a polynomial \(P(\mb{X})=\ell_1(\mb{X})\cdots \ell_k(\mb{X})\sigma(\mb{X})\in\mb{F}_p[\mb{X}]\), where \(\deg\ell_1=\cdots=\deg\ell_k=1\), and \(\sigma(\mb{X})\) is a symmetric polynomial, such that \(\deg P\le d\), \(P|_{\ul{i}}=0\) and \(P|_{\ul{j}}\ne0\).
	\end{conjecture}
	
	Theorem~\ref{thm:mod-p-main} proves a special case of Conjecture~\ref{conj:Z-cl-general-weaker}, but with a stronger assertion about the witness polynomial.  Specifically, Theorem~\ref{thm:mod-p-main} states that for \(d\in\mb{N},\,\ell=\ell_p(d)\), if \(n\ge4p^\ell-1\), then the following is true for any \(E\subseteq[d,n-d]\): if \(j\not\in\zcl_{n,d}(E)\), then there exists a symmetric polynomial \(\sigma(\mb{X})\in\mb{F}_p[\mb{X}]\) such that \(\deg\sigma\le d,\,\sigma|_{\ul{E}}=0\) and \(\sigma|_{\ul{j}}\ne0\).
	\begin{remark}
		The condition \(n\ge4p^\ell-1\) in the assumption of Theorem~\ref{thm:mod-p-main}, where \(\ell=\ell_p(d)\), is simply a requirement for our proof.  We believe the assertion of Theorem~\ref{thm:mod-p-main} is true even for all smaller values of \(n\), though we don't have a proof yet.
	\end{remark}
	
	An immediate corollary of Theorem~\ref{thm:mod-p-main} is the following.
	\begin{corollary}\label{cor:clo-p-main-compute}
		Let \(d\in\mb{N},\,\ell=\ell_p(d)\).  If \(n\ge4p^\ell-1\), then for any \(E\subseteq[d,n-d]\), \(\zcl_{n,d}(E)\) can be computed in \(\poly(n)\) time.
	\end{corollary}
	\begin{proof}
		Let \(d\in[0,n],\,\ell=\ell_p(d)\).  Consider any \(E\subseteq[0,n],\,j\in[0,n]\).  By Theorem~\ref{thm:mod-p-main}, it is enough to show that whether \(j\in\symcl_{n,d}(E)\) can be decided in \(\poly(n)\) time.
		
		Clearly, every symmetric polynomial function, with degree at most \(d\), is a linear combination of \(\sigma_k,\,k\in[0,d]\).  Appealing to Corollary~\ref{cor:sym-digit} (a), the linear system of concern to us is
		\[
		\sum_{k=0}^dc_k\sigma_k(1^i0^{n-i})=0,\quad\tx{for all }i\in E.
		\]
		This is a homogeneous system with \(d+1\le n+1\) variables \(c_k,\,k\in[0,d]\), and \(|E|\le n\) constraints.  The solution space of this system can thus be computed in \(\poly(n)\) time (for instance, by Gaussian elimination).  It is clear that \(j\in\symcl_{n,d}(E)\) if and only if \(\sum_{k=0}^dc_k\sigma_k(1^j0^{n-j})=0\), for every solution \(c_k,\,k\in[0,d]\) of the system.  Therefore, whether \(j\in\symcl_{n,d}(E)\) can be decided in \(\poly(n)\) time.
	\end{proof}

	\subsection{The finite-degree symmetric closure in more detail}\label{subsec:symm-clo-detail}
	
	It is easy to see that again by Theorem~\ref{thm:CN}, while considering the finite-degree symmetric closure \(\symcl_{n,d}\), we can restrict \(d\in[0,n]\).  For any \(d\in[0,n]\) and \(E\subseteq[0,n]\), let \(\mc{S}_{n,d}(E)\) be the set of all symmetric polynomials with degree at most \(d\), that vanish at each point in \(\ul{E}\).  Recall that for any \(E\subseteq[0,n]\), we define \(\symcl_{n,d}(\ul{E})=\mc{Z}(\mc{S}_{n,d}(E))\).  Further, as mentioned earlier, \(\symcl_{n,d}(\ul{E})\) is a symmetric set, and hence we identify (and denote) it by \(\symcl_{n,d}(E)\subseteq[0,n]\).
	
	Let us gather some interesting properties of the finite-degree symmetric closures.  We begin with the following lemma.
	\begin{lemma}\label{lem:translate-clo}
		Let \(d\in[0,n]\) and \(f(\mb{X})=\sum_{u=0}^dc_u\sigma_u(\mb{X})\in\mb{F}_p[\mb{X}]\).  Define
		\begin{align*}
			f^+(\mb{X})&=\sum_{u=0}^d\bigg(\sum_{v=u}^d(-1)^{v-u}c_v\bigg)\sigma_u(\mb{X}),\\
			\tx{and}\quad f^-(\mb{X})&=\sum_{u=0}^d\bigg(\sum_{v=u}^dc_v\bigg)\sigma_u(\mb{X}).
		\end{align*}
		Then
		\begin{enumerate}[(a)]
		\item  \(f(\mb{X})=(f^+)^-(\mb{X})=(f^-)^+(\mb{X})\).
		\item  if \(j,j+1\in[0,n]\), then \(f|_{\ul{j}}=0\) if and only if \(f^+|_{\ul{j+1}}=0\).
		\item  if \(j,j-1\in[0,n]\), then \(f|_{\ul{j}}=0\) if and only if \(f^-|_{\ul{j-1}}=0\).
		\end{enumerate}
	\end{lemma}
	\begin{proof}
		\begin{enumerate}[(a)]
			\item  The assertion follows immediately from the following elementary fact, which is precisely M\"obius inversion (see, for instance, Stanley~\cite[Chapter 3, Section 3.7]{stanley-2011-enumerative-vol1}) for \([0,d]\) with the obvious linear order.
			
			\tsf{Fact.}\quad  Let \(d\in\mb{N}\), and for \(i,j\in[0,d]\), let \(a_{i,j}=1\) if \(i\le j\), and \(a_{i,j}=0\) if \(i>j\).  Then the integer matrices \(M_d=\begin{bmatrix}
				a_{i,j}
			\end{bmatrix}_{i,j\in[0,d]}\) and \(N_d=\begin{bmatrix}
			(-1)^{j-i}a_{i,j}
		\end{bmatrix}_{i,j\in[0,d]}\) both have determinant 1, and are inverses of each other.
	
		\item  Consider any \(f(\mb{X})=\sum_{u=0}^dc_u\sigma_u(\mb{X})\in\mb{F}_p[\mb{X}]\).  By an abuse of notation, consider the integer representatives \(c_u\in[0,p-1],\,u\in[0,d]\), and let \(Q(Z)=\sum_{u=0}^dc_u\binom{Z}{u}\).  Let
		\begin{align*}
			Q^+(Z)&\coloneqq Q(Z-1)&&\\
			&=\sum_{u=0}^dc_u\binom{Z-1}{u}&&\\
			&=\sum_{u=0}^dc_u\bigg(\binom{Z}{u}-\binom{Z-1}{u-1}\bigg)&\tx{by Pascal's triangle, where }\binom{Z}{-1}\coloneqq0&\\
			&=\sum_{u=0}^d\bigg(\sum_{v=u}^d(-1)^{v-u}c_v\bigg)\binom{Z}{u}.&&\\
		\end{align*}
		Let \(j\in[0,n]\) such that \(j+1\in[0,n]\).  Then we have
		\begin{align*}
			\quad&f(x)=0,\quad\tx{for all }x\in\ul{j}&&\\
			\iff\quad&Q(j)\equiv0\modulo{p}&&\\
			\iff\quad&Q^+(j+1)=Q(j)\equiv0\modulo{p}&&\\
			\iff\quad&f^+(x)=0,\quad\tx{for all }x\in\ul{j+1}&&
		\end{align*}
	
		\item  By Item (a), we have \(f(\mb{X})=(f^-)^+(\mb{X})\).  So the assertion follows by Item (b).\qedhere
		\end{enumerate}	
	\end{proof}
		
	The following property of finite-degree symmetric closures then follows quickly.
	\begin{proposition}\label{pro:symcl-prop}
		For any \(d,k,j\in[0,n]\) and \(E\subseteq[0,n]\) such that \(E+k\subseteq[0,n]\), we have \(j\in\symcl_{n,d}(E)\) if and only if \(j+k\in\symcl_{n,d}(E+k)\).
	\end{proposition}
	\begin{proof}
		It is enough to prove that for \(E\subseteq[0,n]\) such that \(E+1\subseteq[0,n]\), we have \(j\in\symcl_{n,d}(E)\) if and only if \(j+1\in\symcl_{n,d}(E+1)\).  Again, it is enough to show that if \(j\in\symcl_{n,d}(E)\), then \(j+1\in\symcl_{n,d}(E+1)\).  The argument for the converse is similar.
		
		Let \(f(\mb{X})\in\mc{S}_{n,d}(E+1)\).  By Lemma~\ref{lem:translate-clo}, we have \(f^-(\mb{X})\in\mc{S}_{n,d}(E)\).  So \(f^-|_{\ul{j}}=0\).  Again by Lemma~\ref{lem:translate-clo}, this implies \(f|_{\ul{j+1}}=(f^-)^+|_{\ul{j+1}}=0\).  This completes the proof.
	\end{proof}
		
	Let us now gather some fairly straightforward lemmas, which are important for our results.
	\begin{lemma}\label{lem:mod-p-translate}
		Let \(d\in[0,n],\,\ell=\ell_p(d)\).
		\begin{enumerate}[(a)]
			\item  For any \(E\subseteq[0,n]\), if \(j\in E\) such that \(j+p^\ell\in[0,n]\), then
			\[
			\symcl_{n,d}(E)=\symcl_{n,d}(n,E\cup\{j+p^\ell\}\setminus\{j\}).
			\]
			\item  For any \(E\subseteq[0,n],\,j\in[0,n]\), we have
			\[
			j\in\symcl_{n,d}(E)\quad\iff\quad j\oplus p^\ell\subseteq\symcl_{n,d}(E).
			\]
			\item  For any \(E\subseteq[d,n-d]\), if \(j\in E\) such that \(j+p^\ell\in[d,n-d]\), then \[\zcl_d(E)=\zcl_d(E\cup\{j+p^\ell\}\setminus\{j\}).\]
			\item  For any \(E\subseteq[d,n-d],\,j\in[d,n-d]\), we have
			\[
			j\in\zcl_{n,d}(E)\quad\iff\quad j\oplus p^\ell\subseteq\zcl_{n,d}(E).
			\]
		\end{enumerate}
	\end{lemma}
	\begin{proof}
		It is easy to note that every symmetric polynomial function, with degree at most \(d\), is a linear combination of \(\sigma_k,\,k\in[0,d]\).  So by Corollary~\ref{cor:sym-digit} (b), for any symmetric polynomial function \(f\) with \(\deg f\le d\), and any \(x,y\in\{0,1\}^n\) with \(|y|\in|x|\oplus p^\ell\), we have \(f(y)=f(x)\).  This concludes the proof of Item (a) and Item (b).
		
		By Theorem~\ref{thm:clo-layer-p}, for any polynomial \(f(\mb{X})\in\mb{F}_p[\mb{X}]\) with \(\deg f\le d\), and any \(x,y\in\{0,1\}^n\) with \(|x|\in[d,n-d],\,|y|\in|x|\oplus p^\ell\in[d,n-d]\), we have \(f(y)=0\) if and only if \(f(x)=0\).  This concludes the proof of Item (c) and Item (d).
	\end{proof}
	
	For any \(E\subseteq[d,n-d]\) and any interval \(I\subseteq\mb{N}\) with \(|I|=p^\ell\), define \(E_I\subseteq I\) as follows:  for any \(j\in I\), define \(j\in E_I\) if \(j+kp^\ell\in E\) for some \(k\in\mb{Z}\).  An immediate consequence of Lemma~\ref{lem:mod-p-translate} is the following observation.
	\begin{observation}\label{obs:LJ}
		Let \(d\in[0,n]\) and \(\ell=\ell_p(d)\).  For any \(E\subseteq[d,n-d]\) and any interval \(I\subseteq[d,n-d]\) with \(|I|=p^\ell\),
		\[
		\symcl_{n,d}(E)=\symcl_{n,d}(E_I)\quad\tx{and}\quad\zcl_{n,d}(E)=\zcl_{n,d}(E_I).
		\]
	\end{observation}
	
	
	\subsection{The main lemmas and proof of the main theorem}
	
	We will now characterize \(\zcl_{n,d}(E)\) for every \(E\subseteq[d,n-d]\), when \(n\) is large.  Let us recall the main theorem.
	\restatethmnow{thm:mod-p-main}{Finite-degree Z-closures of symmetric sets}
	
	We will need the results obtained in Subsection~\ref{subsec:symm-clo-detail}, as well as a couple more.  We will state these results, prove Theorem~\ref{thm:mod-p-main}, and then finish the proofs of the results.
	
	The first result characterizes the duals of a class of linear codes.  This class of codes (called weighted Reed-Muller codes) was first introduced over \(\mb{F}_q^n\) by S\o rensen~\cite{sorensen-1992-weighted-RM}, who also gave a description of their duals.  The result we require is over a finite grid in \(\mb{F}_p^n\), and is a special case of a result by Camps, L\'opez, Matthews and Sarmiento~\cite[Theorem 2.2]{camps-lopez-matthews-2020-monomial-cartesian}.  Consider the indeterminates \(\mb{T}=(T_0,\ldots,T_r)\).  For any \(P(\mb{T})\in\mb{F}_p[\mb{T}]\), define
	\[
	\wdeg_p(P)=\deg P(\sigma_{p^0},\ldots,\sigma_{p_r})=\max\bigg\{\sum_{t=0}^r\alpha_tp^t:\coeff(\mb{T}^\alpha,P)\ne0\bigg\}.
	\]
	\begin{theorem}[{\cite[Theorem 2.2]{camps-lopez-matthews-2020-monomial-cartesian}}]\label{thm:RM-dual}
		Consider the finite grid \(S=S_0\times\cdots\times S_r=[0,p-1]^r\times[0,k]\subseteq\mb{F}_p^{r+1}\), for some \(r\in\mb{N},\,k\in[1,p-1]\). Let \(N=\sum_{i\in[0,r-1]}(|S_i|-1)p^i=(k+1)p^r-1\).  Let
		\[
		\mc{W}(S,d)=\Big\{P\coloneqq\begin{bmatrix}
			P(t)
		\end{bmatrix}_{t\in S}:P(\mb{T})\in\spn\{\mb{T}^\alpha:\alpha\in S\},\,\wdeg_p(P)\le d\Big\}.
		\]
		Then there exists \(\gamma_t\in\mb{F}_p\setminus\{0\}\) for every \(t\in S\), such that
		\begin{align*}
			\mc{W}(S,d)^\perp&=\Big\{\begin{bmatrix}
				\gamma_tQ(t)
			\end{bmatrix}_{t\in S}:Q(\mb{T})\in\spn\{\mb{T}^\beta:\beta\in S\},\,\wdeg_p(Q)\le N-d-1\Big\}\\
			&=\Big\{\diag(\gamma_t:t\in S)\cdot\begin{bmatrix}
				Q(t)
			\end{bmatrix}_{t\in S}:Q\in\mc{W}(S,N-d-1)\Big\}.
		\end{align*}
	\end{theorem}
	
	The second result characterizes, in a special case, when \(1^n\) is in the Z-closure of a symmetric set in \(\{0,1\}^n\).
	\begin{lemma}\label{lem:cl-symcl-lem}
		Let \(n=(k+1)p^r-1\) for some \(k\in[1,p-1]\), and \(d\in[0,n]\).  Then for any \(E\subseteq[0,n]\), \(n\in\zcl_{n,d}(E)\) if and only if \(n\in\symcl_{n,d}(E)\).
	\end{lemma}
	\begin{remark}
		It is easy to see that the assertion of Lemma~\ref{lem:cl-symcl-lem} is not true for general \(n\in\mb{Z}^+\).  In Counterexample~\ref{counterex:not-in-cl}, we see that over \(\mb{F}_2\), \(0\not\in\zcl_{5,2}(\{1,4\})\).  So by Proposition~\ref{pro:clo-prop} (b), we get \(5\not\in\zcl_{5,2}(\{1,4\})\).  Note that \(5=1+2^2\), and so \(5\) is not of the form required in the assumption of Lemma~\ref{lem:cl-symcl-lem}.  Now trivially, we have \(1\in\symcl_{5,2}(\{1,4\})\).  So by Lemma~\ref{lem:mod-p-translate} (b), we get \(5\in1\oplus4\subseteq\symcl_{5,2}(\{1,4\})\), since \(\ell_2(2)=2\).
	\end{remark}
	
	We are now ready to prove our main theorem.
	\begin{proof}[Proof of Theorem~\ref{thm:mod-p-main}]
		Clearly \(\zcl_{n,d}(E)\subseteq\symcl_{n,d}(E)\).  Now let us prove that \(\symcl_{n,d}(E)\subseteq\zcl_{n,d}(E)\).  Since \(n\ge4p^\ell-1\), we get \(n\ge2p^\ell+2d-1\); this means \(|[d,n-d]|\ge2p^\ell\).
		
		Consider any \(j\in\symcl_{n,d}(E)\).  Let \(j'\in(j\oplus p^\ell)\cap[d+p^\ell,d+2p^\ell-1]\).  By Lemma~\ref{lem:mod-p-translate} and Theorem~\ref{thm:clo-layer-p}, it is clear that \(j'\in\symcl_{n,d}(E)\), and further, that it is enough to show \(j'\in\zcl_{n,d}(E)\).  Let \(E'=E_{[j'-p^\ell,j'-1]}\).  By Observation~\ref{obs:LJ}, we have \(j'\in\symcl_d(n,E')\) and we need to show \(j'\in\zcl_{n,d}(E')\).
		
		We have two cases.
		
		\vst
		\noindent\tsf{Case} (i)\quad  \(j'-p^\ell\in E'\).  Since \(j'\ge d+p^\ell\), we have \(j'-p^\ell\ge d\).  Since \(j'\le d+2p^\ell-1\), we have \(j'-p^\ell\le d+p^\ell-1\le d+2p^\ell-1\le n-d\).  Thus \(j'-p^\ell\in[d,n-d]\).  So by Theorem~\ref{thm:clo-layer-p}, we have \(j'\in(j'-p^\ell)\oplus p^\ell\subseteq\zcl_{n,d}(E')\).
		
		\vst
		\noindent\tsf{Case} (ii)\quad  \(j'-p^\ell\not\in E'\).  Then \(E'\subseteq[j'-(p^\ell-1),j'-1]\), and so \(E'-(j'-(p^\ell-1))\subseteq[0,p^\ell-2]\).  Since \(j'\in\symcl_d(n,E')\), we have \(j'\in\symcl_d(E')\), and so by Proposition~\ref{pro:symcl-prop}, \(p^\ell-1=j'-(j'-(p^\ell-1))\in\symcl_d(E'-(j'-(p^\ell-1)))\).  This gives \(p^\ell-1\in\symcl_d(p^\ell-1,E'-(j'-(p^\ell-1)))\).  By Lemma~\ref{lem:cl-symcl-lem}, we get \(p^\ell-1\in\zcl_{p^\ell-1,d}(E'-(j'-(p^\ell-1)))\).  Then by Proposition~\ref{pro:clo-prop} (c), we get \(j'=p^\ell-1+(j'-(p^\ell-1))\in\zcl_{j',d}(E')\).  And finally, by Proposition~\ref{pro:clo-prop} (a), we get \(j'\in\zcl_{n,d}(E')\).
	\end{proof}
	
	We conclude by proving Lemma~\ref{lem:cl-symcl-lem}.  Towards this, let us prove yet another smaller result.
	\begin{lemma}\label{lem:poly-to-poly-sym}
		Let \(n=(k+1)p^r-1\) for some \(k\in[1,p-1]\), and \(d\in[0,n]\).  If \(n\in\symcl_{n,d}(E)\), then for every \(P(\mb{T})\in\mb{F}_p[\mb{T}]\) satisfying \(P(i_0,\ldots,i_r)=0\) for all \(i\in E\), and \(\wdeg_p(P)\le d\),  we have
		\[
		P\Big(\Big(\underbrace{(p-1),\ldots,(p-1)}_{r\tx{ times}},k\Big)\Big)=0.
		\]
	\end{lemma}
	\begin{proof}
		Consider any \(P(\mb{T})\in\mb{F}_p[\mb{T}]\) satisfying \(P(i_0,\ldots,i_r)=0\) for all \(i\in E\), and \(\wdeg_p(P)\le d\).  Then \(P(\sigma_{p^0}(\mb{X}),\ldots,\sigma_{p^r}(\mb{X}))\in\mb{F}_p[\mb{X}]\) is a symmetric polynomial with \(\deg(P(\sigma_{p^0},\ldots,\sigma_{p^r}))=\wdeg_p(P)\le d\).  Further, for any \(x\in\ul{E}\), we get \(P(\sigma_{p^0}(x),\ldots,\sigma_{p^r}(x))=P(|x|_0,\ldots,|x|_r)=0\).  Since \(n\in\symcl_{n,d}(E)\), this implies \(P(\sigma_{p^0}(1^n),\ldots,\sigma_{p^r}(1^n))=0\).  Further, since \(n=(k+1)p^r-1=\sum_{u=0}^{r-1}(p-1)p^u+kp^r\), we have \(\sigma_{p^u}(1^n)=p-1\) for \(u\in[0,r-1]\), and \(\sigma_{p^r}(1^n)=k\).  This completes the proof.
	\end{proof}

	We now have everything in place to prove Lemma~\ref{lem:cl-symcl-lem}.
	\begin{proof}[Proof of Lemma~\ref{lem:cl-symcl-lem}]
		Clearly if \(n\in\zcl_{n,d}(E)\), then \(n\in\symcl_{n,d}(E)\).  
		
		Conversely, suppose \(n\in\symcl_{n,d}(E)\).  Note that since \(n=(k+1)p^r-1\), we have \(\sigma_{p^u}(1^n)=n_u=(p-1)\) for \(u\in[0,r-1]\), and \(\sigma_{p^r}(1^n)=n_r=k\).  By Lemma~\ref{lem:poly-to-poly-sym}, we get
		\[
		P\Big(\Big(\underbrace{(p-1),\ldots,(p-1)}_{r\tx{ times}},k\Big)\Big)=0,
		\]
		for every \(P(\mb{T})\in\mb{F}_p[\mb{T}]\) satisfying \(P(i_0,\ldots,i_r)=0\) for all \(i\in E\), and \(\wdeg_p(P)\le d\).  So by Theorem~\ref{thm:RM-dual}, Fact~\ref{fact:dual-code} and Remark~\ref{rem:GRM-RM-support}, this implies that there exists \(Q(\mb{T})\in\mb{F}_p[\mb{T}]\) such that \(\wdeg_p(Q)\le n-d-1\) and
		\begin{align}
		\Big(\underbrace{(p-1),\ldots,(p-1)}_{r\tx{ times}},k\Big)&\in\supp(Q)\subseteq\Big\{\Big(\underbrace{(p-1),\ldots,(p-1)}_{r\tx{ times}},k\Big)\Big\}\cup\{(i_0,\ldots,i_r):i\in E\}.\label{eq:Q-supp}
		\end{align}
		Define \(R(\mb{X})=Q(\sigma_{p^0}(\mb{X}),\ldots,\sigma_{p^r}(\mb{X}))\).  Then \(\deg R=\wdeg_p(Q)\le n-d-1\), and further by (\ref{eq:Q-supp}), we have
		\[
		1^n\in\supp(R)\subseteq\{1^n\}\cup\ul{E}.
		\]
		By Theorem~\ref{thm:RM-usual}, Fact~\ref{fact:dual-code} and Remark~\ref{rem:GRM-RM-support}, this implies \(n\in\zcl_{n,d}(E)\).
	\end{proof}

	\bibliographystyle{alpha}
	\bibliography{references}

\begin{thebibliography}{HRRY19}

\bibitem[AKV20]{alon-kumar-volk-2020-unbalancing}
Noga Alon, Mrinal Kumar, and Ben~Lee Volk.
\newblock {Unbalancing sets and an almost quadratic lower bound for
  syntactically multilinear arithmetic circuits}.
\newblock {\em Combinatorica}, 40(2):149--178, 2020.
\newblock \url{https://doi.org/10.1007/s00493-019-4009-0}.

\bibitem[Alo99]{alon-1999-combinatorial}
Noga Alon.
\newblock {Combinatorial Nullstellensatz}.
\newblock {\em Combinatorics, Probability and Computing}, 8(1-2):7--29, 1999.
\newblock \url{https://doi.org/10.1017/S0963548398003411}.

\bibitem[Alo20]{alon-2020-problems-extremal-combinatorics-IV}
Noga Alon.
\newblock {Problems and Results in Extremal Combinatorics -- IV}.
\newblock {\em arXiv Preprint}, 2020.
\newblock \url{https://arxiv.org/abs/2009.12692}.

\bibitem[ARS02]{anstee-ronyai-sali-2002-shattering}
Richard~P Anstee, Lajos R{\'o}nyai, and Attila Sali.
\newblock {Shattering news}.
\newblock {\em Graphs and Combinatorics}, 18(1):59--73, 2002.
\newblock \url{https://doi.org/10.1007/s003730200003}.

\bibitem[BD18]{beelen-datta-2018-GHM-RM}
Peter Beelen and Mrinmoy Datta.
\newblock {Generalized Hamming weights of affine Cartesian codes}.
\newblock {\em Finite Fields and Their Applications}, 51:130--145, 2018.
\newblock \url{https://doi.org/10.1016/j.ffa.2018.01.006}.

\bibitem[BE99]{bernasconi-egidi-1999-hilbert-symmetric}
Anna Bernasconi and Lavinia Egidi.
\newblock {Hilbert function and complexity lower bounds for symmetric Boolean
  functions}.
\newblock {\em Information and Computation}, 153(1):1--25, 1999.
\newblock \url{https://doi.org/10.1006/inco.1999.2798}.

\bibitem[BEHL12]{ben-eliezer-hod-lovett-2012-low-degree-polys}
Ido Ben-Eliezer, Rani Hod, and Shachar Lovett.
\newblock {Random low-degree polynomials are hard to approximate}.
\newblock {\em Computational Complexity}, 21(1):63--81, 2012.
\newblock \url{https://doi.org/10.1007/s00037-011-0020-6}.

\bibitem[Bir73]{birkhoff-1973-lattice}
Garrett Birkhoff.
\newblock {\em {Lattice theory}}, volume~25.
\newblock American Mathematical Soc., 1973.

\bibitem[CC97]{cahen-chabert-1997-integer-polynomials}
Paul-Jean Cahen and Jean-Luc Chabert.
\newblock {\em {Integer-valued polynomials}}, volume~48.
\newblock American Mathematical Soc., 1997.

\bibitem[CLMS20]{camps-lopez-matthews-2020-monomial-cartesian}
Eduardo Camps, Hiram~H. L{\'o}pez, Gretchen~L. Matthews, and Eliseo Sarmiento.
\newblock {Monomial-Cartesian codes closed under divisibility}.
\newblock In {\em Finite Fields and their Applications}, pages 199--208. De
  Gruyter, 2020.
\newblock \url{https://doi.org/10.1515/9783110621730-014}.

\bibitem[CLO15]{cox-2015-ideals}
David~A Cox, John Little, and Donal O’Shea.
\newblock {\em {Ideals, Varieties, and Algorithms}}.
\newblock Springer, 2015.
\newblock \url{https://doi.org/10.1007/978-3-319-16721-3}.

\bibitem[CLP17]{croot-lev-pach-2017-progression}
Ernie Croot, Vsevolod~F. Lev, and P{\'e}ter~P{\'a}l Pach.
\newblock {Progression-free sets in \(\mathbb{Z}_4^n\) are exponentially
  small}.
\newblock {\em Annals of Mathematics}, pages 331--337, 2017.
\newblock \url{https://doi.org/10.4007/annals.2017.185.1.7}.

\bibitem[DGM70]{delsarte-goethals-macwilliams-1970-generalized-RM}
P.~Delsarte, J.M. Goethals, and F.J. {Mac Williams}.
\newblock {On generalized Reed-Muller codes and their relatives}.
\newblock {\em Information and Control}, 16(5):403--442, 1970.
\newblock \url{https://doi.org/10.1016/S0019-9958(70)90214-7}.

\bibitem[DKSS13]{dvir-kopparty-saraf-2013-kakeya}
Zeev Dvir, Swastik Kopparty, Shubhangi Saraf, and Madhu Sudan.
\newblock {Extensions to the Method of Multiplicities, with Applications to
  Kakeya Sets and Mergers}.
\newblock {\em SIAM Journal on Computing}, 42(6):2305--2328, 2013.
\newblock \url{https://doi.org/10.1137/100783704}.

\bibitem[Dvi09]{dvir-2009-kakeya}
Zeev Dvir.
\newblock {On the size of Kakeya sets in finite fields}.
\newblock {\em Journal of the American Mathematical Society}, 22(4):1093--1097,
  2009.
\newblock \url{https://doi.org/10.1090/S0894-0347-08-00607-3}.

\bibitem[Dvi12]{dvir-2012-incidence-theorems}
Zeev Dvir.
\newblock {Incidence theorems and their applications}.
\newblock {\em Foundations and Trends® in Theoretical Computer Science},
  6(4):257--393, 2012.
\newblock \url{https://doi.org/10.1561/0400000056}.

\bibitem[EG17]{ellenberg-gijswijt-2017-progression}
Jordan~S. Ellenberg and Dion Gijswijt.
\newblock {On large subsets of \(\mathbb{F}_q^n\) with no three-term arithmetic
  progression}.
\newblock {\em Annals of Mathematics}, pages 339--343, 2017.
\newblock \url{https://doi.org/10.4007/annals.2017.185.1.8}.

\bibitem[FHR09]{felszeghy-hegedus-ronyai-2009-complete-wide}
B\'alint Felszeghy, G\'abor Heged\H{u}s, and Lajos R\'onyai.
\newblock {Algebraic Properties of Modulo \(q\) Complete \(\ell\)-Wide
  Families}.
\newblock {\em Combinatorics, Probability and Computing}, 18(3):309–333,
  2009.
\newblock \url{https://doi.org/10.1017/S0963548308009619}.

\bibitem[FR03]{friedl-ronyai-2003-order-shattering}
Katalin Friedl and Lajos R{\'o}nyai.
\newblock {Order shattering and Wilson's theorem}.
\newblock {\em Discrete Mathematics}, 270(1-3):127--136, 2003.
\newblock https://doi.org/10.1016/S0012-365X(02)00869-5.

\bibitem[FRR06]{felszeghy-rath-ronyai-2006-lex}
B{\'a}lint Felszeghy, Bal{\'a}zs R{\'a}th, and Lajos R{\'o}nyai.
\newblock {The lex game and some applications}.
\newblock {\em Journal of Symbolic Computation}, 41(6):663--681, 2006.
\newblock \url{https://doi.org/10.1016/j.jsc.2005.11.003}.

\bibitem[GK15]{guth-katz-2015-distinct-distances}
Larry Guth and Nets~Hawk Katz.
\newblock {On the Erd{\H{o}}s distinct distances problem in the plane}.
\newblock {\em Annals of mathematics}, pages 155--190, 2015.
\newblock \url{https://doi.org/10.4007/annals.2015.181.1.2}.

\bibitem[Gut16]{guth-2016-polynomial-method}
Larry Guth.
\newblock {\em Polynomial methods in combinatorics}, volume~64.
\newblock American Mathematical Soc., 2016.

\bibitem[Heg10]{hegedus-2010-balancing}
G{\'a}bor Heged{\H{u}}s.
\newblock {Balancing sets of vectors}.
\newblock {\em Studia Scientiarum Mathematicarum Hungarica}, 47(3):333--349,
  2010.
\newblock \url{https://doi.org/10.1556/sscmath.2009.1134}.

\bibitem[Heg21]{hegedus-2021-L-balancing}
G{\'a}bor Heged{\"u}s.
\newblock {$ L $-balancing families}.
\newblock {\em arXiv Preprint}, 2021.
\newblock \url{https://arxiv.org/abs/2105.01526}.

\bibitem[HP98]{heijnen-pellikaan-1998-GHM-Reed-Muller}
P.~Heijnen and R.~Pellikaan.
\newblock {Generalized Hamming weights of \(q\)-ary Reed-Muller codes}.
\newblock {\em IEEE Transactions on Information Theory}, 44(1):181--196, 1998.
\newblock \url{https://doi.org/10.1109/18.651015}.

\bibitem[HR03]{hegedus-ronyai-2003-grobner-complete-uniform}
G{\'a}bor Heged{\H{u}}s and Lajos R{\'o}nyai.
\newblock {Gr{\"o}bner bases for complete uniform families}.
\newblock {\em Journal of Algebraic Combinatorics}, 17(2):171--180, 2003.
\newblock \url{https://doi.org/10.1023/A:1022934815185}.

\bibitem[HR18]{hegedus-ronyai-2018-linear-sperner}
G{\'a}bor Heged{\H{u}}s and Lajos R{\'o}nyai.
\newblock {A note on linear Sperner families}.
\newblock {\em Algebra universalis}, 79(1):1--10, 2018.
\newblock https://doi.org/10.1007/s00012-018-0482-3.

\bibitem[HRRY19]{hrubes-ramamoorthy-rao-2019-balancing}
Pavel Hr{\v{u}}bes, Sivaramakrishnan~Natarajan Ramamoorthy, Anup Rao, and Amir
  Yehudayoff.
\newblock {Lower Bounds on Balancing Sets and Depth-2 Threshold Circuits}.
\newblock In Christel Baier, Ioannis Chatzigiannakis, Paola Flocchini, and
  Stefano Leonardi, editors, {\em 46th International Colloquium on Automata,
  Languages, and Programming (ICALP 2019)}, volume 132 of {\em Leibniz
  International Proceedings in Informatics (LIPIcs)}, pages 72:1--72:14,
  Dagstuhl, Germany, 2019. Schloss Dagstuhl--Leibniz-Zentrum fuer Informatik.
\newblock \url{https://doi.org/10.4230/LIPIcs.ICALP.2019.72}.

\bibitem[KLP68]{kasami-lin-peterson-1968-generalization-RM}
T.~Kasami, Shu Lin, and W.~Peterson.
\newblock {New generalizations of the Reed-Muller codes -- I: Primitive codes}.
\newblock {\em IEEE Transactions on Information Theory}, 14(2):189--199, 1968.
\newblock \url{https://doi.org/10.1109/TIT.1968.1054127}.

\bibitem[KS05]{keevash-sudakov-2005-min-rank-inclusion}
Peter Keevash and Benny Sudakov.
\newblock {Set systems with restricted cross-intersections and the minimum rank
  of inclusion matrices}.
\newblock {\em SIAM Journal on Discrete Mathematics}, 18(4):713--727, 2005.
\newblock \url{https://doi.org/10.1137/S0895480103434634}.

\bibitem[Lin]{lin-unpublished-notes-GRM}
Shu Lin.
\newblock {Unpublished notes on G.R.M. codes}.
\newblock Electrical Engineering Department, University of Hawaii, Honolulu,
  Hawaii 96822.

\bibitem[Luc78]{lucas-1878-II}
Edouard Lucas.
\newblock {Théorie des Fonctions Numériques Simplement Périodiques.
  [Continued]}.
\newblock {\em American Journal of Mathematics}, 1(3):197--240, 1878.
\newblock \url{https://doi.org/10.2307/2369311}.

\bibitem[NW15]{nie-wang-2015-hilbert}
Zipei Nie and Anthony~Y. Wang.
\newblock {Hilbert functions and the finite degree Zariski closure in finite
  field combinatorial geometry}.
\newblock {\em Journal of Combinatorial Theory, Series A}, 134:196--220, 2015.
\newblock \url{https://doi.org/10.1016/j.jcta.2015.03.011}.

\bibitem[S{\o}r92]{sorensen-1992-weighted-RM}
A.B. S{\o}rensen.
\newblock {Weighted Reed-Muller codes and algebraic-geometric codes}.
\newblock {\em IEEE Transactions on Information Theory}, 38(6):1821--1826,
  1992.
\newblock \url{https://doi.org/10.1109/18.165459}.

\bibitem[Sta11]{stanley-2011-enumerative-vol1}
Richard~P. Stanley.
\newblock {\em {Enumerative Combinatorics, Volume 1, Second edition}}.
\newblock Cambridge studies in advanced mathematics. cambridge University
  Press, 2011.
\newblock \url{https://doi.org/10.1017/CBO9781139058520}.

\bibitem[Tao14]{tao-2014-algebraic-combinatorial-geometry}
Terence Tao.
\newblock {Algebraic combinatorial geometry: the polynomial method in
  arithmetic combinatorics, incidence combinatorics, and number theory}.
\newblock {\em EMS Surveys in Mathematical Sciences}, 1(1):1--46, 2014.
\newblock \url{https://doi.org/10.4171/emss/1}.

\bibitem[tur59]{turnbull-1959-newton-correspondences}
{Letter by James Gregory to John Collins, dated 23 November 1670}.
\newblock In H.W. Turnbull, A.~Rupert Hall, Laura Tilling, and J.F. Scott,
  editors, {\em The correspondence of Isaac Newton}, volume I (1665-1675),
  page~46. Cambridge University Press, 1959.

\bibitem[Ven21]{venkitesh-2021-covering-symmetric-sets}
S.~Venkitesh.
\newblock {Covering Symmetric Sets of the Boolean Cube by Affine Hyperplanes}.
\newblock {\em arXiv Preprint}, 2021.
\newblock \url{https://arxiv.org/abs/2107.10385}.

\bibitem[Wei91]{wei-1991-GHM}
V.K. Wei.
\newblock {Generalized Hamming Weights for Linear Codes}.
\newblock {\em IEEE Trans. Inf. Theor.}, 37(5):1412–1418, September 1991.
\newblock \url{https://doi.org/10.1109/18.133259}.

\bibitem[Wil90]{wilson-1990-diagonal-incidence}
Richard~M. Wilson.
\newblock {A diagonal form for the incidence matrices of \(t\)-subsets vs.
  \(k\)-subsets}.
\newblock {\em European Journal of Combinatorics}, 11(6):609--615, 1990.
\newblock \url{https://doi.org/10.1016/S0195-6698(13)80046-7}.

\end{thebibliography}
	

\end{document}